\documentclass{article}
\pdfoutput=1

\usepackage{graphicx}
\usepackage{amsmath}
\usepackage{amssymb, wasysym, amsfonts, latexsym,amsthm}

\newcommand{\set}[1]{\left\{#1\right\}}

\newcommand{\setdef}[2][x]{\left \{#1\,|\,#2 \right \}}

\newcommand{\romb}{\diamondsuit}
\newcommand{\Log}{\mathop{Log}}
\newcommand{\logic}[1]{\mathsf{#1}}

\newcommand{\X}{\mathfrak X}

\newcommand{\pmor}{\twoheadrightarrow}

\newcommand{\Y}{\mathcal Y}

\newcommand{\ve}[1]{\overrightarrow{#1}}

\newcommand{\Nf}{\mathcal{N_\omega}}
\newcommand{\N}{\mathcal{N}}

\newtheorem{theorem}{Theorem}[section]

\newtheorem{proposition}[theorem]{Proposition}

\newtheorem{corollary}[theorem]{Corollary}

\newtheorem{lemma}[theorem]{Lemma}

\theoremstyle{definition}

\newtheorem{definition}[theorem]{Definition}

\theoremstyle{remark}

\begin{document}

  \title{Modal logic of some products of neighborhood frames}

  \author{Andrey Kudinov\\ \small
kudinov\ [at here]\ iitp\ [dot]\ ru \\
\small
Institute for Information Transmission Problems, Russian Academy of
Sciences\\ \small
National Research University Higher School of Economics, Moscow, Russia\\ \small
Moscow Institute of Physics and Technology}

 \maketitle

   \begin{abstract}
  We consider modal logics of products of neighborhood frames and prove that for any pair $\logic{L}$ and $\logic{L'}$  of logics from set $\set{\logic{S4}, \logic{D4}, \logic{D}, \logic{T}}$
modal logic of products of $\logic{L}$-neighborhood frames and $\logic{L'}$-neighborhood
frames is the fusion of $\logic{L}$ and $\logic{L'}$. 
  \end{abstract}
 
\section{Introduction}
Neighborhood frames as a generalization of Kripke semantics for modal logic were invented independently by Dana Scott \cite{Scott1970-SCOAOM-2} and Richard Montague \cite{montague1970}.
Neighborhood semantics is more general than Kripke semantics and in case of normal reflexive and transitive logics coincide with topological semantics. 
In this paper we consider product of neighborhood frames, which was introduced by Sano in \cite{Sano11:AxiomHybriProduMonotNeighFrame}. It is a generalization of product of topological spaces\footnote{``Product of topological spaces'' is a well-known notion in Topology but it is different from what we use here (for details see \cite{benthem:MultimodalLogicsProductsTopologies})} presented in \cite{benthem:MultimodalLogicsProductsTopologies}.

The product of neighborhood frames is defined in the vein of the product of Kripke frames (see \cite{Sege73:Twodimodallogic} and \cite{Sheh78:Twodimodallogic}). But, there are some differences.
In any product of Kripke frames axioms of commutativity and Church-Rosser property are valid. Nonetheless,
as it was shown in \cite{benthem:MultimodalLogicsProductsTopologies}, the logic of the products of all topological spaces is the fusion of logics $\logic{S4} \otimes \logic{S4}$.

In his recent work \cite{Urid11:ModalLogicBitopRatioPlane} Uridia  considers derivational semantics for products of topological spaces.
He proves that the logic of all topological spaces is the fusion of logics $\logic{D4} \otimes \logic{D4}$. 
And in fact $\logic{D4} \otimes \logic{D4}$ is complete w.r.t.~the product of the rational numbers $\mathbb{Q}$. 
Derivational and topological semantics can be considered as a special case of neighborhood semantics. So the result of \cite{Urid11:ModalLogicBitopRatioPlane} and corresponding result for $\logic{S4}$ from \cite{benthem:MultimodalLogicsProductsTopologies} can be obtained as corollaries from the main result of this paper. 

Neighborhood frames are usually considered in the context of non-normal logics since they are usually complete w.r.t. non-normal logics and Kripke frames are not. 
In this paper, however, we will consider only monotone neighborhood frames, that correspond to normal modal logics.
In some sense the results of this paper (and of \cite{benthem:MultimodalLogicsProductsTopologies}, \cite{Urid11:ModalLogicBitopRatioPlane}) shows that neighborhood semantics in some sense is more natural for products of normal modal logics, since there are no need to add extra axioms (at least in some cases).

\section{Language and logics}
In this paper we study propositional modal logic with modal operators. A formula is defined recursively as follows:
$$
\phi ::= p\; |\;\bot \; | \; \phi \to \phi \; | \; \Box_i \phi,
$$
where $p\in \mathrm{PROP}$ is a propositional letter and $\Box_i$ is a modal operator.
Other connectives are introduced as abbreviations: classical connectives are expressed through $\bot$ and $\to$,
dual modal operators $\Diamond_i$ are expressed as follows $\Diamond_i = \lnot\Box_i\lnot$.

\begin{definition} A \emph{normal logic\/} (or \emph{a logic,\/} for short) is
a set of modal formulas closed under Substitution $\left
(\frac{A(p_i)}{A(B)}\right )$, Modus Ponens $\left (\frac{A,\,
A\to B}{B}\right )$ and two Generalization rules $\left
(\frac{A}{\Box_i A}\right)$; containing all
classic tautologies and the following axioms
 $$
 \begin {array}{l}
 \Box_i (p\to q)\to (\Box_i p\to \Box_i q).
 \end{array}
$$

$\logic{K_n}$ denotes \emph{the minimal normal modal logic with $n$ modalities} and $\logic{K} = \logic{K_1}$.
\end{definition}

Let $\logic{L}$ be a logic and let $\Gamma$ be a set of formulas, then $\logic{L} +
\Gamma$ denotes the minimal logic containing $\logic{L}$ and $\Gamma$. If
$\Gamma = \set{A}$, then we write $\logic{L} + A$ rather than $\logic{L}+  \{ A \}$

\begin{definition}
Let $\logic{L_1}$ and $\logic{L_2}$ be two modal logics with one modality $\Box$ then \emph{fusion} of these logics is
\[
\logic{L_1} \otimes \logic{L_2} = K_2 + \logic{L_1}_{(\Box \to \Box_1)} + \logic{L_2}_{(\Box \to \Box_2)};
\] 
where $\logic{L_i}_{(\Box \to \Box_i)}$ is the set of all formulas from $\logic{L_i}$ where all $\Box$ replaced by $\Box_i$.
\end{definition}

In this paper we consider the following four well-known logics:
\begin{align*}
 \logic{D} &= \logic{K} + \Box p \to \romb p;\\
 \logic{T} &= \logic{K} + \Box p \to p;\\
 \logic{D4} &= \logic{D} + \Box p \to \Box\Box p;\\
 \logic{S4} &= \logic{T} + \Box p \to \Box\Box p.\\
\end{align*}

\section{Kripke frames}

The notion of Kripke frames and Kripke models is well known (see \cite{blackburn_modal_2002}), so we only define special kind of frames that we are using in this paper. We can call them \emph{fractal} frames because their basic property is that any cone  is isomorphic to the whole frame. In particular, we consider four types of infinite trees with fixed branching: irreflexive and transitive, reflexive and transitive, irreflexive and non-transitive (any point sees only next level) and reflexive and non-transitive.

\begin{definition}
 Let $A$ be a nonempty set.
\[
A^* = \setdef[ a_1 \ldots a_k]{ a_i \in A}
\]
be the set of all finite sequences of elements from $A$, including the empty sequence $\Lambda$. Elements from $A^*$ we will denote by letters with an arrow (e.g. $\ve{a} \in A^*$) The length of sequence $\ve{a} = a_1 \ldots a_k$ is $k$ (Notation: $l(\ve{a}) = k$) the length of the empty sequence equals 0 ($l (\Lambda) = 0$). Concatenation is denoted by ``$\cdot$'': $(a_1\ldots a_k) \cdot (b_1 \ldots b_l) = \ve{a}\cdot \ve{b} =  a_1\ldots a_k b_1 \ldots b_l$.
\end{definition}

\begin{definition}
  Let $A$ be a nonempty set. We define an infinite frame $F_{in}[A] = (A^*, R)$, such that for $\ve{a}, \ve{b} \in A^*$
\[
\ve{a} R \ve{b} \iff \exists x\in A \left( \ve{b} = \ve{a}\cdot x \right).
\] 
We also defined
\begin{align*}
 F_{rn}[A] &= (A^*, R^r), \hbox{ where $R^r = R \cup Id$ --- reflexive closure};\\
F_{it}[A] &= (A^*, R^*), \hbox{ where $R^* = \bigcup\limits_{i=1}^\infty R^i$ --- transitive closure};\\
F_{rt}[A] &= (A^*, R^{r*}).
\end{align*}

So ``$t$'' stands for transitive, ``$n$'' --- for non-transitive ``$r$'' for reflexive and ``$i$'' for irreflexive.
\end{definition}

The following easy-to-prove proposition shows that frames $F_{\xi \eta}[A]$ (where $\xi \in \set{i, r}$ and $\eta \in \set{t, n})$) are indeed fractal.
\begin{proposition}
 Let $F = F_{\xi \eta}[A] = (A^*, R)$ then
\[
\ve{a} R (\ve{a}\cdot\ve{c}) \iff \Lambda R \ve{c}.
\]
\end{proposition}

\begin{definition}\label{def:F_fusion}
Let $F_1 = F_{\xi_1 \eta_1}[A] = (A^*, R_1)$ be and $F_2 = F_{\xi_2 \eta_2}[B] = (B^*, R_2)$, where $\xi_1, \xi_2 \in \set{i, r}$ and  $\eta_1, \eta_2, \in \set{t, n})$, $A \cap B = \varnothing$, $A = \set{a_1, a_2, \ldots}$ and $B = \set{b_1, b_2, \ldots }$ then we define frame $F_1 \otimes F_2 = (W, R'_1, R'_2)$, as follows
\begin{align*}
 W &= (A \sqcup B)^* \\
\ve{x} R'_1 \ve{y} &\iff \ve{y} = \ve{x}\cdot \ve{z} \hbox{ for some } \ve{z} \in A^* \hbox{ such that } \Lambda R_1 \ve{z} \\
\ve{x} R'_2 \ve{y} &\iff \ve{y} = \ve{x}\cdot \ve{z} \hbox{ for some } \ve{z} \in B^* \hbox{ such that } \Lambda R_2 \ve{z} 
\end{align*}
\end{definition}

\begin{proposition}[\cite{KracWolt91:Propeindepaxiombimodlogic}, \cite{FinSch96}]
 Let $F_1$ and $F_2$ be as in Definition \ref{def:F_fusion} then
\begin{equation}
\Log(F_1 \otimes F_2) = \Log(F_1) \otimes \Log(F_2).
\end{equation} 
\end{proposition}

Let as define four frames: $F_{in}=F_{in}[\omega]$, $F_{rn} = F_{rn}[\omega]$, $F_{it} = F_{it}[\omega]$ and $F_{rt} = F_{rt}[\omega]$

\begin{proposition}\label{prop:compl_Kframes}
For just defined frames
\begin{enumerate}
 \item $\Log(F_{in}) = \logic{D}$;
 \item $\Log(F_{rn}) = \logic{T}$;
 \item $\Log(F_{it}) = \logic{D4}$;
 \item $\Log(F_{rt}) = \logic{S4}$.
\end{enumerate}
\end{proposition}

\section{Neighborhood frames}
In this section we consider neighborhood frames. All definitions and lemmas of this section are well-known and can be found in \cite{Segerberg1971} and \cite{Chellas1980}.

\begin{definition}
 A \emph{(monotone) neighborhood frame} (or an n-frame) is a pair $\X = (X, \tau)$, where $X$ is a nonempty set and $\tau: X \to 2^{2^X}$ such that $\tau(x)$ is a filter on $X$ for any $x$. We call function $\tau$ the \emph{neighborhood function} of $\X$ and sets from $\tau(x)$ we call \emph{neighborhoods of $x$}.
\emph{The neighborhood model} (n-model) is a pair $(\X, V)$, where $\X = (X, \tau)$ is a n-frame and $V: PV \to 2^X$ is a \emph{valuation}.  
In a similar way we define \emph{neighborhood 2-frame} (n-2-frame) as $(X, \tau_1, \tau_2)$ such that $\tau_i (x)$ is a filter on $X$ for any $x$, and a \emph{n-2-model}.
\end{definition}

\begin{definition}
\emph{The valuation of a formula} $\varphi$ at a point of a n-model $M = (\X, V)$ is defined by induction as usual for boolean connectives and for modalities as follows
\[
M, x \models \Box_i \psi \iff \exists V\in \tau_i(x) \forall y\in V (M,y \models \psi).
\]
Formula is valid in a n-model $M$ if it is valid at all points of $M$ (notation $M \models \varphi$).
Formula is valid in a n-frame $\X$ if it is valid in all models based on $\X$ (notation $\X \models \varphi$). 
We write $\X \models L$ if for any $\varphi \in L$, $\X\models \varphi$.
Logic of a class of n-frames $\mathcal{C}$ as $\Log(\mathcal{C}) = \setdef[\varphi]{\X \models \varphi \hbox{ for some } \X \in \mathcal{C}}$.
For logic $L$ we also define $nV(L) =\setdef[\X]{\hbox{$\X$ is an n-frame and } \X\models L}$.
\end{definition}

\begin{definition}
 Let $F = (W, R)$ be a Kripke frame. We define n-frame $\N(F) = (W, \tau)$ as follows. For any $w \in W$
\[
\tau (w) = \setdef[U]{ R(w) \subseteq U \subseteq W}.
\] 
\end{definition}

\begin{lemma}\label{lem:n-frame_from_kframe}
 Let $F = (W, R)$ be a Kripke frame. Then
\[
\Log(\N(F)) = \Log(F).
\]
\end{lemma}

The proof is straightforward. 

\begin{definition} \label{def:bounded_morphism}
 Let $\X = (X, \tau_1, \ldots)$ and $\Y = (Y, \sigma_1, \ldots)$ be n-frames. Then function $f: X \to Y$ is a \emph{bounded morphism} if 
\begin{enumerate}
 \item $f$ is surjective;
 \item for any $x\in X$ and $U \in \tau_i(x)$ $f(U) \in \sigma_i(f(x))$;
 \item for any $x\in X$ and $V \in \sigma_i(f(x))$ there exists $U \in \tau_i(x)$, such that $f(U) \subseteq V$.
\end{enumerate}
In notation $f: \X \pmor \Y$.
\end{definition}

\begin{lemma}\label{lem:pmorhism4n-frame}
 Let $\X = (X, \tau_1, \ldots)$, $\Y = (Y, \sigma_1, \ldots)$ be n-frames, $f: \X \pmor \Y$ $V'$ is a valuation on $\Y$. We define $V(p) = f^{-1}(V'(p))$. Then
\[
\X, V, x \models \varphi \iff \Y, V', f(x) \models \varphi.
\] 
\end{lemma}

The proof is by standard induction on length of formula.

\begin{corollary}
 If $f: \X \pmor \Y$ then $\Log(\Y) \subseteq \Log(\X)$.
\end{corollary}

\begin{definition}
 Let $\X_1=(X_1, \tau_1)$ and $\X_2 =(X_2, \tau_2)$ be two n-frames. Then \emph{the product} of these n-frames is an n-2-frame defined as follows
\[
\begin{array}{l}
\X_1 \times \X_2 = (X_1 \times X_2, \tau_1', \tau_2'),\\
\tau_1'(x_1, x_2) = \setdef[U\subseteq X_1 \times X_2]{\exists V( V \in \tau_1(x_1) \;\&\; V \times \set{x_2}  \subseteq U)},\\
\tau_2'(x_1, x_2) = \setdef[U\subseteq X_1 \times X_2]{\exists V( V \in \tau_2(x_2) \;\&\; \set{ x_1} \times V \subseteq U)}.
\end{array}
\]
\end{definition}

\begin{definition}
For two unimodal logics $\logic{L_1}$ and $\logic{L_2}$ we define \emph{n-product} of them as follows
\[
\logic{L_1} \times_n \logic{L_2} = \Log (\setdef[\X_1 \times \X_2]{\X_1\in nV(L_1) \;\&\; \X_2 \in nV(L_2)})
\]
\end{definition}

Note that $\X_1 \times \X_2$ if we forget about one of its neighborhood functions say $\tau'_2$ then $\X_1 \times \X_2$ will be a disjoint union of $\logic{L_1}$ n-frames. Hence

\begin{proposition}[\cite{Sano11:AxiomHybriProduMonotNeighFrame}]\label{prop:fusin_of_n-frames}
 For two unimodal logics $\logic{L_1}$ and $\logic{L_2}$
\[
\logic{L_1} \otimes \logic{L_2} \subseteq \logic{L_1} \times_n \logic{L_2}.
\]
\end{proposition}

\section{Main construction}

The construction in this section was inspired by \cite{benthem:MultimodalLogicsProductsTopologies}, but it is not a straightforward generalization. 
In case of $\logic{S4 \times_n S4}$ it is, in essence, very similar to the construction in \cite{benthem:MultimodalLogicsProductsTopologies}. 
However, here we operate only with words (finite or infinite), and not with numbers and fractions. It makes proofs shorter and allows us to generalize the results to non-transitive cases.


\newcommand{\st}{\mathop{st}}

Let $F = (A^*, R) = F_{\xi \eta}[A]$ and $0 \notin A$. We define set of ``pseudo-infinite'' sequences
\[
X = \setdef[a_1 a_2 \ldots]{a_i \in A \cup \set{0}\ \&\ \exists N \forall k\ge N (a_k = 0)}.
\] 

Define $f_F:X \to A^*$ which ``fogets'' all zeros. For $\alpha \in X$ such that $\alpha = a_1 a_2 \ldots$ we define 
\begin{align*}
\st(\alpha) &= \min \setdef[N]{\forall k\ge N (a_k = 0)};\\
\alpha|_k &= a_1 \ldots a_k;\\ 
U_k(\alpha) &= \setdef[\beta \in X]{\alpha |_m = \beta |_m \ \&\ f_F(\alpha) R f_F(\beta), \hbox{ where } m = \max(k, \st(\alpha))}.
\end{align*}

\begin{lemma}\label{lem:U_m-base}
$U_k (\alpha) \subseteq U_m (\alpha)$ whenever $k \ge m$. 
\end{lemma}
\begin{proof}
Let $\beta \in U_k (\alpha)$. Since $\alpha |_k = \beta |_k$ and $k \ge m$ then $\alpha |_m = \beta |_m$. Hence, $\beta \in U_m (\alpha)$.
\end{proof}

\begin{definition}\label{def:n-frame_from_fframe}
Due to Lemma \ref{lem:U_m-base} sets $U_n(\alpha)$ forms a filter base. So we can define
\begin{align*}
\tau (\alpha) &- \hbox{the filter with base } \setdef[U_n(\alpha)]{n \in \omega};\\
\Nf(F) &= (X,\tau) \hbox{ --- is \emph{the n-frame based on} $F$.} 
\end{align*}
\end{definition}

\begin{lemma}\label{lem:bmorphism_wframe2nframe}
 Let $F = (A^*, R) = F_{\xi, \eta}[A]$ then
\[
f_F : \Nf(F) \pmor \N(F).
\]
\end{lemma}

\begin{proof}
From now on in this proof we will omit the subindex in $f_F$.
Let $\Nf(F) = (X, \tau)$. 
Since for any $\ve{x} \in A^*$ sequence $\ve{x} \cdot 0^\omega \in X$ and $f(\ve{x} \cdot 0^\omega) = \ve{x}$ then $f$ is surjective.

Assume, that $x\in X$ and $U \in \tau(x)$. We need to prove that $R(f(x)) \subseteq f(U)$. 
There is $m$ such that $U_m(x) \subseteq U$ and since $f(U_m(x)) = R(f(x))$ then
\[
R(f(x)) = f(U_m(x)) \subseteq f(U).
\]

Assume that $x\in X$ and $V$ is a neighborhood of $x$, i.e.\/ $R(f(x)) \subseteq V$. We need to prove that there exists $U \in \tau(x)$, such that $f(U) \subseteq V$. As $U$ we take $U_m(x)$ for some $m \ge st(x)$, then
\[
f(U_m(x))=R(f(x)) \subseteq V. 
\]
\end{proof}

\begin{corollary}\label{cor:logicNf}
 For frame $F = F_{\xi \eta}[A]$ $\Log(\Nf(F)) \subseteq \Log(F)$.
\end{corollary}

\begin{proof}
 It follows from Lemmas \ref{lem:n-frame_from_kframe}, \ref{def:bounded_morphism} and \ref{lem:bmorphism_wframe2nframe}
\[
\Log(\Nf(F)) \subseteq \Log(\N(F)) = \Log(F). 
\]
\end{proof}

\begin{proposition}\label{prop:logicsOf_omega_Nf_framas}
 Let $F_{in} = F_{in}[\omega]$, $F_{rn} = F_{rn}[\omega]$, $F_{it} = F_{it}[\omega]$ and $F_{rt} = F_{rt}[\omega]$ then
\begin{enumerate}
 \item\label{it1} $\Log(\Nf(F_{in})) = \logic{D}$;
 \item\label{it2} $\Log(\Nf(F_{rn})) = \logic{T}$;
 \item\label{it3} $\Log(\Nf(F_{it})) = \logic{D4}$;
 \item\label{it4} $\Log(\Nf(F_{rt})) = \logic{S4}$.
\end{enumerate}
\end{proposition}
\begin{proof}
 In all these cases the inclusion from left to right is covered by Corollary \ref{cor:logicNf} and Proposition \ref{prop:compl_Kframes}.

Let us check the inclusion in converse direction.

(\ref{it1}). It is easy to check that n-frame $\X = (X, \tau) \models \logic{D}$ iff for each $x \in X$ $\varnothing \notin \tau(x)$. For $\Nf(F_{in})$ and $\Nf(F_{it})$ it obviously true.

(\ref{it2}). It is easy to check that n-frame $\X = (X, \tau) \models \logic{T}$ iff $x \in U \in \tau(x)$ for each $x$ and $U$. For $\Nf(F_{rn})$ and $\Nf(F_{rt})$ it is obviously true.

(\ref{it3}) and (\ref{it4}). It is well-known (see e.g.{} \cite{Hans03:Monotmodallogic}) that $\X = (X, \tau) \models \Box p \to \Box \Box p$ iff for each $U \in \tau(x)$ $\setdef[y]{ U \in \tau(y)} \in \tau (x)$. 
Indeed, it follows from the fact that for any $y \in U_m (x)$ and any $k$ $U_k(y) \subseteq U_m(x)$.
\end{proof}

Let $F_1 = (A^*, R_1) = F_{\xi_1\eta_1}[A]$ and $F_2 = (B^*, R_2) = F_{\xi_2\eta_2}[B]$ we assume that $A \cap B = \varnothing$, $A = \set{a_1, a_2, \ldots}$ and $B = \set{b_1, b_2, \ldots }$. Consider the product of n-frames $\X_1 = (X_1, \tau_1) = \Nf(F_1)$ and $\X_2 = (X_2, \tau_2) = \Nf(F_2)$
\[
\X = (X_1 \times X_2, \tau_1', \tau_2') = \Nf(F_1) \times_n \Nf(F_2).
\]

We define function $g: \X_1 \times \X_2 \to (A \cup B)^*$ as follows. For $(\alpha, \beta) \in \X_1 \times \X_2$, such that $\alpha = x_1 x_2 \ldots$ and $\beta = y_1 y_2 \ldots$, $x_i \in A \cup \set{0}$, $y_j \in B \cup \set{0}$, we define 
$g(\alpha, \beta)$ to be the finite sequence which we get after eliminating all zeros from the infinite sequence
$x_1y_1x_2y_2 \ldots$.

\begin{lemma}\label{lem:main}
 Function $g$ defined above is a bounded morphism: $g: \X \pmor \N(F_1 \otimes F_2)$.
\end{lemma}

\begin{proof} 
Let $\ve{z} = z_1 z_2 \ldots z_n \in (A \cup B)^*$. Define for $i\le n$
\[
x_i = \left\{
\begin{array}{l}
z_i, \hbox{ if $z_i \in A$};\\
0, \hbox{ if $z_i \notin A$};  
\end{array}
\right.\qquad
y_i = \left\{
\begin{array}{l}
z_i, \hbox{ if $z_i \in B$};\\
0, \hbox{ if $z_i \notin B$}.  
\end{array}
\right.
\]
Let $\alpha = x_1 x_2 \ldots x_n 0^\omega$ and $\beta = y_1 y_2 \ldots y_n 0^\omega$ then $g(\alpha, \beta) = \ve{z}$.
Hence $g$ is surjective.

The next two conditions we check only for $\tau_1$ and for $\tau_2$ it is similar.
Assume, that $(\alpha, \beta) \in X_1 \times X_2$ and $U \in \tau_1(\alpha, \beta)$. We need to prove that $R_1'(g(\alpha, \beta)) \subseteq g(U)$.
There is $m > \max\set{st(\alpha), st(\beta)}$ such that $U'_m(\alpha) \times \set{\beta} \subseteq U$ and since $g(U'_m(\alpha) \times \set{\beta}) = R_1'(g(\alpha, \beta))$ then
\[
R_1'(g(\alpha, \beta)) = g(U'_m(\alpha)\times \set{\beta}) \subseteq g(U);
\]
where $U'_m(\alpha)$ is the corresponding neighborhood from $\X_1$.

Assume that  $(\alpha, \beta) \in X_1 \times X_2$ and $R_1'(g(\alpha, \beta)) \subseteq V$. We need to prove that there exists $U \in \tau'_1(\alpha, \beta)$, such that $g(U) \subseteq V$. As $U$ we take $U'_m(\alpha)\times \set{\beta}$ for some $m  > \max\set{st(\alpha), st(\beta)}$, then
\[
g(U_m' (\alpha) \times \set{\beta}) = R_1'(g(\alpha, \beta)) \subseteq V.
\]
\end{proof}

\begin{corollary}\label{cor:logicNftimesNf}
 Let $F_1 = (A^*, R_1) = F_{\xi_1\eta_1}[A]$ and $F_2 = (B^*, R_2) = F_{\xi_2\eta_2}[B]$ then 
 $\Log(\Nf(F_1) \times_n \Nf(F_2)) \subseteq \Log(F_1) \otimes \Log(F_2)$.
\end{corollary}

It immediately follows from Lemmas \ref{lem:main}, \ref{lem:pmorhism4n-frame} and Proposition \ref{prop:fusin_of_n-frames}.

\begin{corollary}\label{cor:main}
 Let $F_1, F_2 \in \set{F_{in}, F_{rn}, F_{it}, F_{rt}}$ then 
 $\Log(\Nf(F_1) \times_n \Nf(F_2)) = \Log(F_1) \otimes \Log(F_2)$.
\end{corollary}

\begin{proof}
 The left-to-right inclusion follows from Corollary \ref{cor:logicNftimesNf}.

 To prove right-to-left inclusion we notice that due to Proposition \ref{prop:logicsOf_omega_Nf_framas} $\Log(\Nf(F_i)) = \Log(F_i)$ ($i=1,\,2$) and due to Proposition \ref{prop:fusin_of_n-frames}
\[
\Log(\Nf(F_1)) \otimes \Log(\Nf(F_2)) \subseteq \Log(\Nf(F_1)) \times_n \Log(\Nf(F_2)).
\]
\end{proof}

\section{Completeness results}

\begin{theorem}
Let  $\logic{L_1}, \logic{L_2} \in \set{\logic{S4}, \logic{D4}, \logic{D}, \logic{T}}$ then
\[
\logic{L_1} \times_n \logic{L_2} = \logic{L_1} \otimes \logic{L_2}.
\]
\end{theorem}

\begin{proof}
 Logics $\logic{L_1} = \Log(F_1)$ and $\logic{L_2}= \Log(F_2)$ for some $F_1, F_2 \in \set{F_{in}, F_{rn}, F_{it}, F_{rt}}$.
By Corollary \ref{cor:main} 
\[
\logic{L_1} \times_n \logic{L_2} = \Log(\Nf(F_1) \times_n \Nf(F_2)) =  \Log(F_1) \otimes \Log(F_2) = \logic{L_1} \otimes \logic{L_2}.
\]
\end{proof}

The following fact was proved in \cite{benthem:MultimodalLogicsProductsTopologies}.

\begin{corollary}\label{cor:S4timesS4}
 Let $\X = (\mathbb{Q}, \tau)$ where $\mathbb{Q}$ is the set of rational numbers and $\tau$ is based on the standard topology on $\mathbb{Q}$, i.e.\/ $\tau(x) = \setdef[U]{\exists V (\hbox{$x \in V$ is open and } V \subseteq U)}$. Then
\[
\Log(\X \times_n \X) = \logic{S4} \otimes \logic{S4}.
\]
\end{corollary}

\begin{proof}
 Let  $\Nf(F_{rt}) = (X, \tau)$. We can assume that $X = \setdef[\ve{x}\cdot 0^\omega]{\ve{x}\in \mathbb{Z}}$. 
Note that $X$ is a countable set and neighborhood function $\tau$ is based on topology generated by the lexicographical order $<_l$ on $X$. According to the classical result of Cantor, since the lexicographical order on $X$ is dense, $(X, <_l)$ isomorphic to $(\mathbb{Q}, <)$ (see ) and corresponding topological spaces are homeomorphic. Hence, 
\[
\Log(\X \times_n \X) = \Log( \Nf(F_{rt}) \times \Nf(F_{rt})) = \logic{S4} \otimes \logic{S4}. 
\]
\end{proof}

The following fact was announced\footnote{The talk at the conference was very detailed, but to my knowledge, the full proof has not been published yet.} in \cite{Urid11:ModalLogicBitopRatioPlane}.

\begin{corollary}\label{cor:D4timesD4}
 Let $\X = (\mathbb{Q}, \tau)$ where $\mathbb{Q}$ is the set of rational numbers and $\tau$ is based on the standard topology on $\mathbb{Q}$, i.e.\/ $\tau(x) = \setdef[U]{\exists V (\hbox{$x \in V$ is open and } V \setminus \set{x} \subseteq U)}$. Then
\[
\Log(\X \times_n \X) = \logic{D4} \otimes \logic{D4}.
\]
\end{corollary}

\begin{proof}
 Let $\X = (\mathbb{Q}, \tau)$, where $\tau$ is based on derivation operator in $\mathbb{Q}$. Since $(X, <_l)$ isomorphic to $(\mathbb{Q}, <)$  then
\[
\Log(\X \times_n \X) = \Log( \Nf(F_{it}) \times \Nf(F_{it})) = \logic{D4} \otimes \logic{D4}. 
\]
\end{proof}

\section{Conclusion}

There are several ways to continue research. One of them is to try and extend the technique to other logics (e.g.~$\logic{K}$). 
The other way is to add the third modality which corresponds to the following neighborhood function $\tau' (x, y) = \setdef[U]{\exists V_1 \in \tau_1 (x) \,\&\, \exists V_2 \in \tau_2(y) \left( V_1 \times V_2 \subseteq U \right)}$. Similar construction was considered in \cite{benthem:MultimodalLogicsProductsTopologies} for the topological semantic.


\bibliographystyle{plain}
\bibliography{aiml12}

\end{document}